\documentclass[a4paper,11pt]{amsart}
\usepackage{amsmath}  
\usepackage{amssymb}
\usepackage{amsthm}  
\usepackage{amscd}  
\usepackage{enumerate}
\theoremstyle{plain}
\newtheorem{thm}{Theorem}[section]
\newtheorem{prop}[thm]{Proposition}
\newtheorem{lemma}[thm]{Lemma}
\newtheorem{cor}[thm]{Corollary}

\theoremstyle{definition}
\newtheorem{defn}{Definition}

\theoremstyle{remark}
\newtheorem{rem}{Remark}

\numberwithin{equation}{section}

\newcommand{\R}{\mathbf R}
\newcommand{\C}{\mathbf C}
\newcommand{\SU}{\mathrm{SU}}

\newcommand{\SO}{\mathrm{SO}}
\newcommand{\G}{\mathrm G}
\newcommand{\K}{\mathrm K}

\title[Holomorphic isometric embeddings]
{Holomorphic isometric embeddings of the projective line into quadrics}

\author[Macia, Nagatomo, Takahashi]{Oscar Macia, Yasuyuki Nagatomo, Masaro Takahashi}

\address[OM]{Department of Geometry and Topology, Faculty of Mathematical Sciences, UNIVERSITY OF VALENCIA,
C.Dr Moliner, 50, Burjassot, 46100, Valencia, SPAIN}
\email{oscar.macia@uv.es}

\address[YN]{Department of Mathematics, MEIJI UNIVERSITY, 
Higashi-Mita, Tama-ku, Kawasaki-shi, Kanagawa 214-8571, JAPAN} 
\email{yasunaga@meiji.ac.jp}

\address[MT]{Department of General Education, KURUME NATIONAL COLLEGE OF TECHNOLOGY, Kurume, Fukuoka 830-8555, JAPAN.}
\email{masaro@GES.kurume-nct.ac.jp}


\begin{document}

\maketitle
\begin{abstract}
We discuss holomorphic isometric embeddings of
the projective line into quadrics using the generalisation
of the theorem of do Carmo--Wallach in \cite{Na-13} to provide a description
of their moduli spaces up to image and gauge--equivalence.
Morover, we show rigidity of the real standard map from the projective
line into quadrics.
\end{abstract}

\section{Introduction}

A harmonic map from a Riemannian manifold into a Grassmannian manifold is characterized by  a triple
composed by (a) a vector bundle, (b) a space of sections of this bundle and (c) a Laplace operator, as shown in \cite{Na-13}.
This characterisation can be regarded as a generalisation of a theorem of T. Takahashi \cite{TTaka}
which proves that an isometric immersion of a Riemannian manifold in Euclidean space is an 
eigenvector for the Laplacian iff it is a minimal immersion in some Euclidean sphere,
the energy density being related to the corresponding eigenvalue. In its generalised form
the standard sphere in Euclidean space is identified with the appropriate Grassmannian, and the isometric immersion (the position 
vector) is considered as a section of the universal quotient bundle over it. Then, the energy density of the mapping is related to the
mean curvature operator (defined in \cite{Na-13},\S 2)
of the pull--back of the universal quotient bundle.  Hence, the theorem of Takahashi can
be reformulted from the viewpoint of vector bundles and their spaces of sections, leading to the 
characterisation through the aforesaid triple.\\ \indent
In this viewpoint, a vector bundle and a finite dimensional space of sections induce a map into a Grassmannian.
A celebrated example of such induced map is Kodaira's embedding of an algeberaic manifold into complex projective space \cite{kodaira}, 
which is induced by a holomorphic line bundle and the space of holomorphic sections.\\

The pioneering work of Takahashi found application in do Carmo and Wallach undertaking of the classification of
minimal (isometric) immersions of spheres into spheres \cite{DoC-Wal}. Their staggering result reveals that starting at dimension three
and constant scalar curvature four, there are continuous families of image inequivalent  minimal immersions of spheres into
standard spheres of dimension high enough. Each of these families is parametrised by a moduli space depending on no less than eighteen parameters,
yielding a lower bound for the moduli dimension. Alternatively, for the lower dimension and constant scalar curvature cases  
minimal immersions of spheres into spheres are unique whenever they exist. In this case, the moduli collapses to a point and the map is said to be rigid.\\
\indent A key role in do Carmo--Wallach theory is played by a symmetric semipositive--definite linear operator interweaving minimal immersions. 
Finding the space of the image inequivalent operators ammounts to describe the moduli space, 
an endevour which is dealt successfully with representation--theoretic 
techniques.\\

The generalisation of the theorem of Takahashi in \cite{Na-13} allowed the second--named author to achieve altogether a generalisation of the results
of do Carmo--Wallach. In its general form the induced map into a Grassmannian by the aforementioned harmonic triple is naturally equipped
with a family of Hermitian semipositive--definite operators determining the moduli space.\\ \indent
Uniqueness of the associated Hermitian operator reduces the moduli to a single point granting rigidity of the induced map.
An important illustration of such behaviour is  Bando and Ohnita's result \cite{Ban-Ohn} stating the rigidity of the minimal immersion 
of the complex projective line into complex projective spaces, originally proved employing twistor methods. Another instance is rigidity
of holomorphic isometric embeddings between complex projective spaces, which is part of Calabi's result \cite{Cal}. Both conclusions can be
given a unified treatment applying the generalisation of the theorem of do Carmo--Wallach.\\

Closer to the vector bundle viewpoint is Toth's analysis of polynomial minimal immersions between projective spaces \cite{Tot} where the spaces of harmonic polynomials in complex space
are used to define polynomial maps between spheres and the Hopf fibration to get a map between complex projective spaces.  Representation theory of unitary groups is then put in
practice to determine a lower bound for the moduli dimension.\\ \indent 
It is remarkable that neither \cite{DoC-Wal} nor \cite{Tot} do require the vector bundle viewpoint of \cite{Na-13} since
in this later sense respond for straightforward situations: in the original do Carmo and Wallach construction 
the associated vector bundle would be  the trivial bundle; Toth's result follows from considering  a complex line bundle
with canonical connection.\\

The study of harmonic maps from the complex projective line into complex quadrics has previously been pursued 
in different ways, e.g., in \cite{ChiZhe}, \cite{FJXX}, \cite{LiYu} and \cite{Wolf}.
In the present article we apply the generalisation of the theorem of do Carmo--Wallach to the study of holomorphic isometric embeddings
of the projective line into quadrics  to give a description of the moduli spaces up to image and gauge--equivalence. 
The authors would like to emphasize that contrary to \cite{DoC-Wal,Tot} their  approach allows  to compute the exact
dimension of the moduli spaces. 
Also notice that by  considering complex quadrics as target instead of complex projective spaces as in \cite{Cal}, positive--dimensional
moduli spaces appear.\\

The article is organised as follows: In \S 2 we introuce the required preliminaries to the theory culminating in the statement of the 
generalisation of the theorem of do Carmo--Wallach (theorem \ref{GenDW}) as developed in \cite{Na-13}. The following two sections are technical in
nature. First,  \S 4 gives an account of certain relevant spectral formulae for real $\SU(2)$ representations. After that, \S 5 deals with the study of the 
space of Hermitian/symmetric operators 
and banking on 
  \S 4,  concludes with a detailed description of its various subspaces (proposition \ref{DecH})
which encode the information about the moduli spaces.  Applications of the theory first appear in \S 6: by showing that the space of 
symmetric operators yielding the moduli space restricts  to a single  point, we prove rigidity for the real standard map (theorem \ref{rigidity}). Finally, in Sections 7 and 8 the moduli spaces up to image and gauge equivalence are introduced
and described (theorems \ref{gmod1} and \ref{imod}). \\

\paragraph{\bf Acknowledgements.} \emph{
The authors are grateful to Prof. Y. Ohnita for useful conversations, and would also like 
to thank Kyushu University for its hospitality at several stages during the completion of this article.
The first and third--named authors would like  to thank  also  Meiji University where part of this work was developed.
The work of the first--named author was supported by the Spanish Agency of Scienctific and Technological Research (DGICT) and FEDER project 
MTM2013--46961--P. The work of the second--named author was  supported by JSPS KAKENHI Grant Number 26400074.}

\section{Preliminaries}

In this section we give a short account of results concerning vector bundles endowed with fibre--metrics
and connections needed to state a version of the generalisation of the theorem of do Carmo--Wallach (theorem \ref{GenDW}), whose
implications will be applied later in this article. In essence, the theorem establishes a correspondence
between \emph{geometric} gauge--theoretic information associated to a special class of harmonic mappings and \emph{algebraic} representation--theoretic properties of some Hermitian operators.\\ \indent
Let us start by introducing the required geometric background.\\

Let $W$ be a complex (resp. real, resp. real oriented) $N$--dimensional vector space
and $Gr_p(W)$ the complex (resp. real, resp. real oriented) Grassmannian  $p$--planes in $W$.
Generically, $\underline W$ will stand for the total space of a trvial vector bundle $\underline W\to B$
with fibre $W$ over some specified base manifold $B.$
Denote by $\underline{W}\to Gr_p(W)$ the trivial bundle of fibre $W$ over $Gr_p(W).$ Then,
there is a natural bundle injection $i_S:S\to \underline{W}$ of the \emph{tautological
vector bundle} $S\to Gr_p(W)$ into the aforementioned trivial bundle. The \emph{universal quotient bundle} $Q\to Gr_p(W)$
is defined by the exacteness of the sequence $0 \to S \to \underline{W} \to Q \to 0.$ Denote by $\pi_Q$ the
natural projection $\underline W \to Q$ and use it to regard $W$ as a subspace of $\Gamma(Q),$ the space
of sections of the universal quotient bundle. 

By fixing a Hermitian (resp. symmetric) inner product on $W$ the tautological and universal quotient 
bundles $S,Q\to Gr_p(W)$ inherit a fibre--metric, and can be given canonical connections and second fundamental forms in
the sense of Kobayashi \cite{Kob}.\\

Suppose $V \to M$ is a complex (resp. real, resp. real oriented) vector bundle of rank $q$ and 
consider a  $N$--dimensional space of sections $W\subset \Gamma(V).$ By definition of $\underline W\to M,$ 
there is a bundle homomorphism 
$ev:\underline W \to V,$ called \emph{evaluation}, defined by
$(x,t)\mapsto t(x)$  for all $t\in W,x\in M.$
The vector bundle $V\to M$ is said to be {\it globally generated by} $W$ 
if the evaluation is surjective. 
Under this hypothesis, there is a map $f:M \to Gr_p(W)$,  where $Gr_p(W)$ is a complex (resp. real, resp. real oriented) Grassmannian 
and $p=N-q$, defined by 
\[ f(x):=\text{Ker}\,ev_x=\left\{t\in W \,\vert \, t(x)=0 \right\},\]
where $ev_x\equiv ev(x,\cdot).$  The map $f$ is said to be {\it induced by}
$(V\to M,\;W)$, or simply by $W$ if the vector bundle $V\to M$ is specified.  

Notice that, by the definition
of induced map, 
$V\to M$ can be 
\emph{naturally identified} with $f^*Q \to M.$  
Therefore, given a smooth map $f:M\rightarrow Gr_p(W),$ it
can be regarded as the induced map determined the by the couple $(f^*Q\to M,\;W).$ 
If the inclusion
$W\to\Gamma(f^{\ast}Q)$ is injective, we say that the map $f$ is \emph{full}.

Moreover, assume $M$ to be Riemannian and $V\to M$ to be equipped with a fibre--metric and a connection. 
From  these data a Laplace operator acting on sections can be defined.\\ \indent The model special case  is
that in which $M$ is a compact reductive homogeneous space $G/K$ (where $G$ is a compact
Lie group and $K$ is a closed subgroup of $G$), and $V\to M$ is a homogeneous 
complex (resp. real) vector bundle of rank $q$, ie., $V\cong G\times_K V_0$ where $V_0$ is a $q$--dimensional 
complex (resp. real) $K$--module. If additionally $V_0$ admits a $K$--invariant Hermitian (resp. symmetric)
inner product, $V\to M$ inherits a 
$G$--invariant Hermitian (resp. symmetric) fibre--metric.

By reductivity, $V\to M$ is equipped with a canonical connection too, the one for which the
horizontal subspace on the principal $K$--bundle $G\to M$ is given by the complement $\mathfrak m$ to $\mathfrak k=L(K)$ in $\mathfrak g=L(G).$ \\ \indent
Using the Levi--Civita connection and the canonical connection, $\Gamma(V)$ can be decomposed into eigenspaces of
the Laplacian
each being a finite--dimensional not necessarily irreducible $G$--module and equipped with a $G$--invariant $L^2$--inner product. Then, 
we say that the induced map by $(V\to M,\;W)$ is \emph{standard} 
if a $G$--submodule $W\subseteq W_\mu$ globally generates the bundle, where $W_\mu$ is the eigenspace of the Laplacian with eigenvalue $\mu.$\\ \indent
Evidently, the definition of standard map
reaches beyond
the special homogeneous case. The spaces of sections inducing standard maps have the following interesting property
which will be useful later:
\begin{lemma}\label{subsp}\cite{Na-13}
Let $W$ be a $G$--subspace of $W_{\mu}$.  If $W$ globally generates $V\to G/K$,  then $V_0$ can be regarded as a subspace of $W$. 
\end{lemma}

Denote by $U_0$ the orthogonal complement of $V_0$ in $W.$ Then, the induced standard map $f_0:M\to Gr_p(W)$ is expressed as
\[  f_0([g])=gU_0 \subset W, \]
for all $[g]\in G/K,$ and is $G$--equivariant.\\

Notice that, besides its assummed fibre--metric and connection, $V\to M$ is endowed with a secondary couple 
of fibre--metric and connection innherited from the natural identification  $V\cong f^*Q,$ i.e., 
the fibre--metric and canonical connection on $Q\to Gr_p(W)$ pulled--back to $f^*Q\to M.$\\ \indent 
In general, these structures do not need to be gauge equivalent unless the splitting $W=U_0\oplus^\perp V_0$ satisfies
extra conditions:

\begin{lemma}\cite{Na-13}
The pull--back connection is gauge equivalent to the canonical connection if and only if 
\[\mathfrak{m}V_0 \subset U_0.\]
\end{lemma}
As we shall have opportunity to see, this condition will turn out to be very relevant for the developement
of the remaining theory.\\

For a standard map $f,$ the second fundamental forms $\mathrm{H}$ and $\mathrm{K}$ of the tautological and universal quotient bundles 
respectively can
be assembled together in the \emph{mean curvature operator of $f$}, a section of $\mathrm{End}(f^*Q)$ defined in \cite{Na-13}, \S 2, as
\[ A =\sum_{i=1}^n \mathrm{H}_{df(e_i)} \mathrm{K}_{df(e_i)}\]
where $\{e_i\}_{1\leq i\leq n}$ is an orthonormal basis of the tangent space of $M.$

\begin{lemma}\label{stharm}\cite{Na-13}
If a $G$--module $W\subseteq W_\mu$ globally generates $V\to M$ and satisfies the condition $\mathfrak{m}V_0 \subset U_0$, then the
 standard map $f_0:M\to Gr_p(W)$ is harmonic with constant energy density $e(f_0)=q\mu$ and the mean 
curvature operator proportional to the identity $A=-\mu  Id_V.$
\end{lemma}

Let us introduce the two increasingly stronger equivalence relations up to which we shall later define moduli spaces of maps: 
Let $f_1$ and $f_2:M\to Gr_p(W).$ Then $f_1$ is called {\it image equivalent} to $f_2$
if there exists an isometry $\phi$ of $Gr_p(W)$ such that $f_2=\phi\circ f_1.$
Furthermore, denote by $\tilde \phi$ the bundle isomorphism of $Q\to Gr_p(W)$ which covers the isometry $\phi$ of $Gr_p(W).$ 
Then, the pair $(f_1,\phi_1)$ is said to be {\it gauge equivalent} to $(f_2, \phi_2)$, 
where $\phi_i:V\to f_i^{\ast}Q (i=1,2)$ are bundle isomorphisms,  
if there exists an isometry $\phi$ of $Gr_p(W)$ such that 
$f_2=\phi\circ f_1$ and $\phi_2=\tilde \phi\circ \phi_1$.\\

Aside from the geometric background, some algebraic preliminaries regarding Hermitian operators are needed.\\

Let $G$ be a compact Lie group, $W$ a complex 
(resp. real) 
representation  of $G$ together with an invariant Hermitian
(resp. inner) 
product $(\;,\;)_{_W}$ and
denote by $\mathrm{H}(W)$ 
(resp. $\mathrm{S}(W)$)
the set of Hermitian 
(resp. symmetric) 
endomorphisms of $W.$ 
We equip $\mathrm{H}(W)$ 
(resp. $\mathrm{S}(W)$) with a $G$--invariant inner product 
$(A,B)_{_H}=\text{trace}\,AB$, for $A,B \in \mathrm{H}(W)$. 
Define a Hermitian (resp. symmetric) operator $\mathrm{H}(u,v)$ (resp. $\mathrm{S}(u,v)$) for $u$, $v \in W$ as 
\[ 
\mathrm{H}(u,v):=\frac{1}{2}\left\{u\otimes (\cdot, v)_{_W} + v\otimes (\cdot, u)_{_W} \right\}\qquad \left(resp.\;\mathrm{S}(u,v)\right)
\]
If $U$ and $V$ are subspaces of $W$, we define a real subspace  $\mathrm{H}(U,V) \subset \mathrm{H}(W)$ (resp. $\mathrm{S}(U,V)\subset\mathrm{S}(W)$) spanned by $\mathrm{H}(u,v)$ 
(resp. $\mathrm{S}(u,v)$) where $u\in U$ and $v\in V$. 
In a similar fashion, $\G\mathrm{H}(U,V)$  (resp. $\G\mathrm{S}(U,V)$) denotes the subspace of $\mathrm{H}(W)$ (resp. $\mathrm{S}(W)$) spanned by 
$g\mathrm{H}(u,v)$ (resp. $g\mathrm{S}(u,v)$).

Now we have all the needed ingredients to introduce a version of the generalistion \cite{Na-13}
 of the theorem of do Carmo--Wallach \cite{DoC-Wal} for holomorphic maps which we include for convenience of the reader:

\begin{thm}\label{GenDW}
Let $M=G/K$ be a compact irreducible Hermitian symmetric space with decomposition 
$\mathfrak{g}=\mathfrak{k}\oplus \mathfrak{m}$. 
We fix a complex homogeneous line bundle 
$L=G\times_{K} V_0 \to G/K$ 
with an invariant metric $h_L$ and the canonical connection $\nabla.$ 
We regard $L\to G/K$ as a real vector bundle with complex structure $J_L$. 

Let 
$f:M \to Gr_n(\mathbf R^{n+2})$ be a full holomorphic map 
satisfying the following two conditions: 

\noindent{\rm (i)} 
The pull--back bundle 
$f^{\ast}Q \to M$ 
with the pull--back metric, connection and complex structure 
is gauge equivalent to $L\to M$ with $h_L$, $\nabla$ and $J_L$. 

\noindent{\rm (ii)} 
The mean curvature operator $A \in \Gamma(\text{\rm End}\,L)$ of $f$ 
is expressed as 
$-\mu Id_L$ with some real positive number $\mu$, 
and so $e(f)=2\mu $. 

Then we have the space of holomorphic sections $W$ of $L \to M$ which is 
also an eigenspace of the Laplacian with eigenvalue 
$\mu$ equipped with $L^2$--inner product $(\cdot,\cdot)_{_W}$ 
induced from $L^2$--Hermitian inner product. 
Regard $W$ as a real vector space with $(\cdot,\cdot)_{_W}$. 
Then, there exists a semipositive symmetric endomorphism 
$T\in \text{\rm End}\,(W)$ 
such that the pair $(W,T)$ satisfies the following four conditions:

\noindent {\rm (I)} The vector space $\mathbf R^{n+2}$ is a subspace of $W$ with the inclusion 
$\iota:\mathbf R^{n+2} \to W$ preserving the orientation and 
$L\to M$ is globally generated by $\mathbf R^{n+2}$.

\noindent {\rm (II)} 
As a subspace, $\mathbf R^{n+2}=\text{\rm Ker}\,T^{\bot}$ and 
the restriction of $T$  is a positive symmetric transformation of $\mathbf R^{n+2}$. 

\noindent {\rm (III)} 
The endomorphism $T$ satisfies 
\begin{equation}\label{HDW 2}
\left(T^2- Id_W, \G\mathrm{H}(V_0, V_0)\right)_{_H}=0, 
\qquad
\left(T^2, \G\mathrm{H}(\varrho(\mathfrak m)V_0, V_0)\right)_{_H}=0.
\end{equation}

\noindent {\rm (IV)} 
The endomorphism $T$ provides a holomorphic embedding of $Gr_n(\mathbf R^{n+2})$ into $Gr_{n^{\prime}}(W)$, 
where $n^{\prime}=n+\text{\rm dim}\,\text{\rm Ker}\,T$
and also provides a bundle isomorphism $\phi:L\to f^{\ast}Q$. 

Then, $f:M \to Gr_{p}(\mathbf R^{n+2})$ 
can be expressed as 
\begin{equation}\label{HDW5} 
f\left([g]\right)=\left(\iota^{\ast}T\iota \right)^{-1}\left(f_0\left([g]\right)\cap \text{\rm Ker}\,T^{\bot}\right), 
\end{equation}
where $\iota^{\ast}$ denotes the adjoint operator of $\iota$ under the induced inner product on 
$\mathbf R^{n+2}$ from $(\cdot,\cdot)_W$ on $W$ and $f_0$ is the standard map by $W$. 
Such two pairs $(f_i, \phi_i)$, $(i=1,2)$ are gauge equivalent if and only if 
$
\iota_1^{\ast}T_1\iota_1=\iota_2^{\ast}T_2\iota_2, 
$
where $(T_i,\iota_i)$ correspond to $f_i$ $(i=1,2)$ under the expression in \eqref{HDW5}, respectively. 

Conversely, 
suppose that a vector space $\mathbf R^{n+2}$, the space of holomorphic sections 
$W \subset \Gamma(V)$ regarded as real vector space and 
a semipositive symmetric endomorphism 
$T\in \text{\rm End}\,(W)$ 
satisfying 
conditions {\rm (I)}, {\rm (II)} 
and {\rm (III)} are given.  
Then we have a unique holomorphic embedding of $Gr_n(\mathbf R^{n+2})$ into $Gr_{n^{\prime}}(W)$ 
and 
the map $f:M \to Gr_{n}(\mathbf R^{n+2})$ defined by \eqref{HDW5}
is a full holomorphic map into $Gr_n(\mathbf R^{n+2})$ 
satisfying conditions (i) and (ii) with bundle isomorphism $L\cong f^{\ast}Q$. 
\end{thm}

\begin{proof}
This is obtained by a combination of theorems 5.16 and 5.20 in \cite{Na-13}. 
\end{proof}

\begin{rem} \emph{Conditions (i) and (ii) in the theorem are named respectively  \emph{gauge} 
and \emph{Einstein--Hermitian} conditions; the later is often denoted simply as \emph{EH} condition for short.}
\end{rem}

\section{Generalized do Carmo--Wallach theory for holomorphic isometric embeddings}

The aim of this section is to introduce holomorphic isometric embeddings from $\mathbf C P^1$
into $Gr_n(\mathbf R^{n+2})$ and to show that they  satisfy the hypothesis of the generalized version of
do Carmo--Wallach theory developed in \cite{Na-13} and given in theorem \ref{GenDW}. \\

We realize a complex quadric of $\mathbf CP^{n+1}$ as 
a real oriented Grassmannian $Gr_n(\mathbf R^{n+2}).$
Then the universal quotient bundle has a holomorphic bundle structure.
Notice that the curvature two--form $R$ of the canonical connection on the quotient bundle is the fundamental two--form
$\omega_Q$ on $Gr_n(\mathbf R^{n+2})$ up to a constant multiple 
\[R=-2\pi{\sqrt{-1}}\omega_Q.\]
Denote by $\omega_0$ the fundamental two--form on $\mathbf CP^1$. 
When $R_1$ denotes the curvature two--form of the canonical connection on $\mathcal O(1)\to \mathbf CP^1$, 
we also have 
$R_1=-2\pi{\sqrt{-1}}\omega_0$

\begin{defn}
Let $f:\mathbf CP^1 \to Gr_n(\mathbf R^{n+2})$ be a holomorphic embedding. 
Then $f$ is called an isometric embedding of degree $k$ if 
$f^{\ast}\omega_Q=k\omega_0$ (and so, $k$ must be a positive integer). 
\end{defn}

\begin{lemma}\label{necessary}
Let $f:\mathbf C P^1 \to Gr_n(\mathbf R^{n+2})$ be a holomorphic embedding. 
Then $f$ is an isometric embedding of degree $k$ if and only if
the pull--back bundle $f^{\ast}Q \to \mathbf C P^1$ with the pull--back connection is 
gauge equivalent to $\mathcal O(k) \to \mathbf C P^1$ 
with the canonical connection. 
\end{lemma}

\begin{proof}
If the degree of the isometric embedding $f$ equals $k,$ the pull--back of 
the universal quotient bundle is holomorphically isomorphic to the holomorphic 
line bundle of degree $k$ on $\mathbf C P^1$ (by uniqueness of the holomorphic bundle structure), which by homogeneity admits a unique Einstein--Hermitian 
structure up to homotheties of the fibre--metric. Uniqueness of the Einstein--Hermitian connection
yields the result.\\ \indent
Conversely, if the pull--back of the universal quotient bundle is holomorphically isomorphic as Einstein--Hermitian bundle to the holomorphic 
line bundle, the pull--back fibre--metric and the Einstein--Hermitian connection coincide up to homothety,
and the statement in the lemma follows. 
\end{proof}

\begin{lemma}
Let $f:\mathbf C P^1 \to Gr_n(\mathbf R^{n+2})$ be a holomorphic isometric embedding of degree $k.$  Then,
the mean curvature operator $A \in \Gamma(V)$ of $f$ is the identity on 
$V$ up to a  negative real constant.
\end{lemma}

\begin{proof} It is well--known that every 
holomorphic section $t$ of $\mathcal O(k)\to \mathbf CP^1$ satisfies $\Delta t - K_{EH} t =0,$
where the Laplacian is defined through a compatible connection; $K_{EH}$ is the mean curvature operator 
arising from the Hermitian structure in the sense of Kobyashi \cite{Kob}. 
Since the canonical connection is the Einstein--Hermitian connection, 
$K_{EH}=\mu  Id$. 

On the other and, the generalisation of theorem of Takahashi \cite{Na-13} yields that 
$\Delta t+At=0$ for $t \in \mathbf R^{n+2}$. 
Regard $\mathbf R^{n+2}$ as a subspace
of $H^0\left(\mathbf CP^1,\mathcal O(k)\right);$ 
then it globally generates $\mathcal{O}(k)\to\C P^1,$ therefore $K_{EH}=-A,$ and the 
lemma follows.\end{proof}

These two lemmas amount to saying that the holomorphic embedding $f$ is isometric iff
it satisfies the gauge condition, and then the EH
condition is automatically satisfied. 
Hence we can apply theorem \ref{GenDW} to obtain the moduli space $\mathcal M_k$
of holomorphic isometric embeddings of degree $k$ by the gauge equivalence of maps.

\begin{rem}
\emph{Unlike the case of holomorphic isometric embeddings, for general harmonic maps and minimal immersions 
the EH condition is independent of the gauge condition. 
We shall discuss harmonic maps and minimal immersions satisfying gauge and EH conditions 
in a forthcoming paper.}
\end{rem}

\section{Real representations of $\text{SU}(2)$}

Let $S^{k}\mathbf C^2$ be the $k$--th symmetric power of the standard, complex representation  
of $\text{SU}(2)$.
Since $\mathbf C^2$ has an invariant quaternionic structure $j$, $S^{2k}\mathbf C^2$ 
inherits an invariant real structure $\sigma=j^{2k},$ while $S^{2k+1}\mathbf C^2$ is equipped with an 
induced invariant quaternionic structure $j^{2k+1}$.
We shall denote the standard, real representation of $\text{\rm SO}(3)$  by $\mathbf R^3$ 
and its $l$--th symmetric power by $S^l\mathbf R^3$ .\\ \indent
We start by pointing out a fundamental relation between real irreducible representations of
$\text{SU}(2)$ and $\text{SO}(3):$

\begin{lemma}\label{rltnshp}
For $k\geqq 2$, $S^k\R^3$ admits the following decomposition:
\[ S^k\mathbf R^3=S^k_0\mathbf R^3 \oplus S^{k-2}\mathbf R^3\]
where \[S^k_0\R^3=(S^{2k}\C^2)^\R\]
is the real irreducible $\SU(2)$--representation defined as the $\sigma$--invariant real subspace 
of $S^{2k}\C^2.$
\end{lemma}

\proof
Consider the $k$--th symmetric product of $\R^3.$ Since $\SO(3)$ admits an invariant symmetric
two--tensor, $S^k\R^3$ is not irreducible but splits into $ S^k_0\R^3\oplus S^{k-2}\R^3,$
for the canonical decomposition of the trace homomorphism.  
Counting dimensions leads to ${\rm dim}\;S^k_0\R^3=2k+1.$\\ \indent 
Recall the following elementary identity between real $\SU(2)$ and $\SO(3)$ representations: 
$\R^3$ is the real form of $S^2\C^2$ induced by the involution $\sigma,$ 
denoted by $(S^2\C^2)^\R$;  conversely, $S^2\C^2$ is the complexification
of $\R^3.$ Complexification of $S^k\R^3$ is given by \[\C\otimes_\R S^k\R^3 = S^k(\C\otimes_\R\R^3)\cong S^k(S^2\C^2)\]
which is a component in $\otimes^k(S^2\C^2).$ Applying Clebsch--Gordan spectral formula, 
the top term $S^{2k}\C^2\subset S^k(S^2\C)\subset\otimes^k S^2\C^2$ is a complex irreducible $\SU(2)$--representation of complex dimension $2k+1,$ 
equiped with an invariant real structure. Therefore, $(S^{2k}\C^2)^\R\subset S^k\R^3$ is a real irreducible $\SU(2)$--representation.\\ \indent
 However by Schur's lemma $S^{2k}\C^2\cap S^{k-2}(S^2\C^2)=\{0\},$ the second factor being the complexification of
$S^{k-2}\R^3\subset S^k\R^3,$ and the result follows. 
\qed \\

Once we have identified the real irreducible representations of $\SU(2)$ we would like to have a spectral
formula for the tensor product. To that end, it is enough to restrict to the real stable subspace
of the real structure 

\begin{lemma}\label{spctrl}
When $k\geqq l$, then 
\begin{equation}\label{CGR}
S^k_0\mathbf R^3\otimes S^l_0\mathbf R^3 = 
\bigoplus_{r=0}^{2l} S^{k+l-r}_0\mathbf R^3.
\end{equation}
\end{lemma}

\proof Complexification of $S^k_0\R^3$ is given by $S^{2k}\C^2.$ The spectral formula for the tensor product
of complex irreducible representations of $\SU(2)$ is well known,
\begin{equation}\label{CG} S^{2k}\C^2\otimes_\C S^{2l}\C^2 = \bigoplus_{r=0}^{min(k,l)}S^{2k+2l-2r}\C^2\end{equation}
restricting to the real stable subspaces proves the lemma.\qed\\

Regardless of the existence of a real structure any comlpex irreducible $\SU(2)$--representation 
can be interpreted as a real not necessarily irreducible representation by considering its underlying
$\R$--vector space such that $(S^n\C^2)|\R \cong (\R^{2n+2},J)$ where $J$ is a complex structre.\\ \indent  
In the absence of a real structure, i.e., for odd $n,$ this is a real irreducible $\SU(2)$--representation; when $n$ 
is even, this is a real reducible $\SU(2)$--representation further splitting into the  stable subspaces for the action
of the induced real structure $(S^n\C^2)^\R,$ which in this case are the truly real irreducible representations of $\SU(2).$\\ \indent

For later use in the study of real standard maps, it will be useful to have a spectral formula for the decomposition of 
tensor products of the underlying $\R$--vector spaces of a  given complex $\SU(2)$--representations into real irreducible ones. 
This is how such result can be achieved:  Given $\R$--vector spaces $U,V$ underlying some complex irreducible $\SU(2)$--representations 
(therefore not necessarily irreducible as real representations) consider their complexification $U^\C:=\C\otimes_\R U,$ etc. 
These are complex not necessarily irreducible $\SU(2)$--representations, which easily abscind into irreducible ones. 
Taking the tensor product $U^\C\otimes_\C V^\C$ we might apply the known 
spectral formulae for decomposing products of complex $\SU(2)$--representations. The tensor product inherits a real structure, such that
real irreducible representations can be constructed by the general formula
\begin{equation}\label{RC}U\otimes_\R V = \left(U^\C\otimes_\C V^\C\right)^\R.\end{equation}
Applying the preceding general argument we have the following

\begin{lemma}
When we regard $S^{2k}\mathbf C^2$ as a real representation space $\mathbf R^{4k+2}$
of $\text{SU}(2)$, the second symmetric power $S^2\mathbf R^{4k+2}$ has the
following irreducible decomposition:
\begin{equation}\label{idsp1} S^2\mathbf R^{4k+2}=3\left( \bigoplus_{r=0}^{k}\; S^{2k-2r}_0\mathbf R^3 \right) \oplus
\left( \bigoplus_{r=0}^{k-1} \;S^{(2k-1)-2r}_0\mathbf R^3 \right).\end{equation}
When we regard $S^{2k+1}\mathbf C^2$ as a real representation space $\mathbf R^{4k+4}$ 
of $\text{SU}(2)$, the second symmetric power $S^2\mathbf R^{4k+4}$ has the 
following irreducible decomposition:
\begin{equation}\label{idsp2} S^2\mathbf R^{4k+4}=3 \left(\bigoplus_{r=0}^k S^{(2k+1)-2r}_0\mathbf R^3 \right) \oplus
 \left(\bigoplus_{r=0}^{k-1} S^{2k-2r}_0\mathbf R^3 \right).\end{equation}
\end{lemma}

\begin{proof} 
Considered as its underlying $R$--vector space, $S^{2k}\C^2$ is simply $\R^{4k+2}.$ This is not
an irreducible since $S^{2k}\C^2$ is equipped with an invariant real structure, but decomposes as 
two copies of $S^k_0\R^3$ by lemma \ref{rltnshp}. Therefore, using the spectral formula
in lemma \ref{spctrl},
\begin{eqnarray}\label{R4k+2}
\R^{4k+2}\otimes_\R \R^{4k+2}  =   2 S^k_0 \R^3 \otimes_\R 2 S^k_0\R^3 = 4 \bigoplus_{r=0}^{2k} S^{2k-r}_0\R^3.
\end{eqnarray}
In order to obtain the symmetrised tensor product we need to subtract the alternating terms $\Lambda^2\R^{4k+2}.$\\
Complexifying, $\Lambda^2=\C\otimes_\R\Lambda^2\R^{4k+2}=\Lambda^2(S^{2k}\C^2\oplus S^{2k}\overline \C^2).$
Write $\Lambda^{2,0}$ for $\Lambda^2(S^{2k}\C^2)$ etc. Then, $\Lambda^2=\Lambda^{2,0}\oplus \Lambda^{1,1}\oplus \Lambda^{0,2}$ where, at the real level,
$\Lambda^{2,0} \cong \Lambda^{0,2}.$ Using (\ref{CG}) we easily identify
$(\Lambda^{2,0}\oplus \Lambda^{0,2})^\R = 2(\bigoplus_{r=0}^{2k}S^{4k-2(2r+1)}\C^2)^\R,$ hence from (\ref{CGR})
\begin{equation}\label{lmbd20R}
(\Lambda^{2,0}\oplus\Lambda^{0,2})^\R \cong 2  \bigoplus_{r=0}^{k-1} S^{2k-(2r+1)}_0\R^3 
\end{equation}
For $(\Lambda^{1,1})^\R$ recall that $(S^{2k}\C^2)\cong \C\otimes_\R S^k_0\R^3\cong S^{2k}\overline \C^2;$ therefore,
substitution of $U=V=S^k_0\R^3$ in (\ref{RC}) leads to $(\Lambda^{1,1})^\R\cong \otimes^2 S^k_0\R^3;$ using, (\ref{CGR}): 
\begin{equation}\label{lmbd11R} 
(\Lambda^{1,1})^\R \cong \bigoplus_{r=0}^{2k} S^{2k-r}_0\R^3.
\end{equation}
Subtracting (\ref{lmbd20R}) and (\ref{lmbd11R}) from (\ref{R4k+2}) yields the desired result (\ref{idsp1}).\\

When regarding $S^{2k+1}\C^2$ as a $\R$--vector space, we identify it with $\R^{4k+4}$ which is an irreducible
$\SU(2)$--representation. Its complexification is reducible under the action of the inherited complex strucutre
leading to $(\R^{4k+4}) ^\C\cong S^{2k+1}\C^2\oplus S^{2k+1}\overline\C^2.$ We compute the tensor product 
to be:

\begin{equation}\label{R4k+4} 
\R^{4k+4}\otimes_\R\R^{4k+4} \cong   4\bigoplus_{r=0}^{2k+1} S^{2k+1-r}_0\R^3 
\end{equation}

applying Eqns. (\ref{CGR}-\ref{RC}) to simplify the expression :
\[ \left(\left( S^{2k+1}\C^2\oplus S^{2k+1}\overline\C^2\right)\otimes_\C \left(S^{2k+1}\C^2\oplus S^{2k+1}\overline\C^2\right)\right)^\R . \]
Again, to get the result we need to subtract the subspace of alternating elements. Applying the same arguments as in
the previous case (where now $\Lambda^{2,0}=\Lambda^2(S^{2k+1}\C^2),$ etc.), we are lead to:
\begin{eqnarray}
\label{first} (\Lambda^{2,0}\oplus\Lambda^{0,2})^\R &=&  2 \bigoplus_{r=0}^{2k} S^{2k-r}_0\R^3,\\
\label{second} (\Lambda^{1,1})^\R &=&\bigoplus_{r=0}^{2k+1} S^{2k+1-r}_0\R^3.
\end{eqnarray}
Equation (\ref{idsp2}) is achieved by subtracting (\ref{first}) and (\ref{second}) from (\ref{R4k+4}).\end{proof}

\section{The space of Hermitian/Symmetric endomorphisms of $\SU(2)$--representations}

In the previous section we summarised some results on real $\SU(2)$--modules but
in order to apply the generalised do Carmo--Wallach theory we need a deeper understanding
of the space of symmetric endomorphisms of these representations. \\ \indent
In the present section we describe how the space of symmetric endomorphisms of a real
irreducible $\SU(2)$--module splits into irreducible components.\\

Let $W$ be a $\C$--vector space with a Hermitian inner product. 
When the coefficient field on $W$ is restricted to the  field of real numbers, 
we denote by $W_{\mathbf R}$ the resulting $\R$--vector space with the complex structure $J$. 
The Hermitian inner product induces a symmetric inner product on $W_{\mathbf R},$
simply by taking the real part.\\ \indent 
If $\mathrm{H}(W)$ denotes the $\R$--vector space of Hermitian endomorphisms on $W$ and 
$\mathrm{S}(W_{\mathbf R})$ the $\R$--vector space of all symmetric endomorphisms on $W_{\mathbf R},$ 
it follows from  general considerations above that $\mathrm{H}(W) \subset \mathrm{S}(W_{\mathbf R}),$
while $\C$--linearity of $A\in \mathrm{H}(W)$ is reflected in $\mathrm{S}(W_\R)$ by commutation of $A$ and $J.$\\

Suppose that $W$ has a real (resp. quaternionic) structure denoted by $\sigma$
compatible with the Hermitian inner product.
Then $\mathrm{H}(W)$ has a regular action of $\sigma$ such that $A\mapsto \sigma A \sigma$,
where $A$ is a Hermitian endomorphism.
Hence, we can define the subspaces $\mathrm{H}_{\pm}(W)$ of $\mathrm{H}(W)$ as the set of invariant/anti--invariant Hermitian endomorphisms with respect to $\sigma$. 
The action of $\sigma$ extends to $\mathrm{S}(W_{\mathbf R})$ in the obvious way.

\begin{lemma}
If $A \in \mathrm{H}_+(W)$, then real endomorphisms $\sigma A$ and $J\sigma A$ are  symmetric endomorphisms on $W_{\mathbf R}$.
\end{lemma}

\begin{proof}
For simplicity, we assume that $\sigma$ is a real strucutre. 
If $\sigma$ is a quaternionic structure the proof goes along the same lines. 

Let $A\in \mathrm{H}_+(W)$ and so $\sigma A = A\sigma.$ 

Denote the Hermitian inner product on $W$ by $(\;,\;),$ with the convention in which
it is $\C$--linear in the first argument, and let $\langle\;,\;\rangle$ be the induced
symmetric inner product on $W_\R.$  Then, for $u,v\in W\cong W_\R,$
\begin{eqnarray*}
\langle \sigma A u,v \rangle & = & \mathrm{Re} (\sigma A u,v) = \mathrm{Re}\overline{(Au,\sigma v)} =  \mathrm{Re}(u,A\sigma v) \\
                              & =& \mathrm{Re} (A\sigma v,u)  =  \mathrm{Re}(\sigma A v,u) = \langle \sigma A v,u\rangle
\end{eqnarray*} Therefore, $\sigma A \in \mathrm{S}(W_\R).$  The proof for $J\sigma A$ is analogous. \end{proof}

Notice that $\sigma A$ (resp. $J\sigma A)$ above is not an Hermitian operator since $\sigma$ is by definition conjugate--linear. We put 
\[ \sigma \mathrm{H}_+(W):=\left\{\sigma A \,|\, A\in \mathrm{H}_+(W) \right\}\subset \mathrm{S}(W_\R),\] 
\[J\sigma \mathrm{H}_+(W):=\left\{J\sigma A \,|\, A\in \mathrm{H}_+(W) \right\} \subset \mathrm{S}(W_\R).\]
A characterization of these subspaces is given as follows:

\begin{lemma}
Let $B$ be a symmetric endomorphism of $W_\R.$ Then,
\begin{enumerate}
\item $B$ belongs to $\sigma \mathrm{H}_+(W)$ if and only if  $JB=-BJ$ and $\sigma B \sigma =B$;
\item $B$ belongs to $J\sigma \mathrm{H}_+(W)$ if and only if $JB=-BJ$ and $\sigma B \sigma =-B$.
\end{enumerate}
\end{lemma}

\begin{proof}
For $B$ in $\sigma \mathrm{H}_+(W)$ (resp. in $J\sigma \mathrm{H}_+(W))$  there exists $A\in \mathrm{H}_+(W)$ such that $B=\sigma A$ (resp. $J\sigma A$). 
Write $BJ,\;\sigma B\sigma$ in terms of $A;$ commutation  relations  for $A,\;J,\;\sigma$ yield the implications.\\ \indent
Conversely, condition $JB=-BJ$ implies that $B$ is not Hermitian; hence, $A:=\sigma B$ (resp. $A:=J\sigma B$) is, for commutation relations between $J,\sigma$
lead to $AJ=JA.$ Ivariance under the regular action of $\sigma$ on $\mathrm{H}(W)$ shows $A\in \mathrm{H}_+(W),$ therefore  $B$ belongs to $\sigma \mathrm{H}_+(W)$ (resp. $J\sigma \mathrm{H}_+(W)$). \end{proof}

Subspaces $\sigma \mathrm{H}_+(W)$ and $J\sigma \mathrm{H}_+(W)$ are orthogonal with respect to the inherited inner product on $\mathrm{S}(W_\R),$
Then, counting dimensions we have:

\begin{cor}\label{DecSYM}
We have a decomposition of $\mathrm{S}(W_{\mathbf R})$: 
$$
\mathrm{S}(W_{\mathbf R})=\mathrm{H}_+(W) \oplus \mathrm{H}_-(W) \oplus \sigma \mathrm{H}_+(W) \oplus J\sigma \mathrm{H}_+(W).
$$
\end{cor}

\begin{rem}\label{complex}
\emph{As a result, the orthogonal complement of} $\mathrm{H}(W)$ \emph{in} $\mathrm{S}(W_{\mathbf R})$ \emph{has the induced complex structure.}
\end{rem}
Applying Corollary \ref{DecSYM} to the real representations of $\text{SU}(2)$ discussed in \S 4 yields:

\begin{prop}\label{DecH}
\begin{eqnarray*}
\mathrm{H}_+(S^{2k}\mathbf C^2)=\bigoplus_{r=0}^{k} S^{2k-2r}_0\mathbf R^3,&     & \mathrm{H}_-(S^{2k}\mathbf C^2)=\bigoplus_{r=0}^{k-1} S^{2k-1-2r}_0\mathbf R^3\\
\mathrm{H}_+(S^{2k+1}\mathbf C^2)=\bigoplus_{r=0}^{k} S^{2k+1-2r}_0\mathbf R^3,&  & \mathrm{H}_-(S^{2k+1}\mathbf C^2)=\bigoplus_{r=0}^{k} S^{2k-2r}_0\mathbf R^3.
\end{eqnarray*}
\end{prop}

\section{Rigidity of the real standard map}                                               

Essential at this stage is to prove proposition \ref{pre_rigidity} (and its real invariat counterpart proposition \ref{pre_rigidity2}). 
This is a technical result that states in short that if each factor
in the normal decomposition of a $G$--module $W$ 
is inequivalent as a $K$--representation to any other factor, there is a certain $G$--orbit in $\mathrm{H}(W)$ which 
contains all class--one representations of $(G,K).$ Since in our case $\mathrm{H}(W)$ itself is composed of class--one
representations only, the $G$--orbit mentioned earlier fills $\mathrm{H}(W).$ \\ \indent
The proposition has
a practical reading: the Hermitian/symmetric operators parametrising the moduli spaces belong to the orthogonal complement
in $\mathrm{H}(W)$ to the aforesaid $G$--orbit, but in the present situation this space is null.  Therefore
the induced map will be rigid. We use this information to study the real standard map, the outcome naming the section (theorem \ref{rigidity}).\\

A detailed description of the normal decomposition can be found in \cite{DoC-Wal}. Let us sketch the central
ideas: 
Consider the situation described in 
\S 2, i.e., 
 $W\subset \Gamma(V)$ is a space of sections of the vector bundle $V \to M,\;M=G/K$ associated to the principal homogeneous bundle $G\to G/K$ with standard fibre the irreducible $K$--representation 
$V_0\subset W.$ Furthermore, suppose $V\to M$ to be equipped with its canonical connection. Let $f: G/K\to Gr_p(W)$ be the corresponding induced map by $(V\to M,\;W).$
The space of sections $W$ splits into $V_0$ and its orthogonal complment $N_0=U_0.$ Assume the
condition of lemma 2.1, i.e., $\mathfrak{m}V_0\subset U_0$ such that the canonical connection and the pull--back connection coincide.\\ \indent  
From now on our considerations will be restricted at a point $o\in M$ for the sake of simplicity. 
The second fundamental form $\K$ at $o\in M$ is an element  of $T_o^*M\otimes V_0^*\otimes U_0$ such that
for all $X\in T_oM,\;v\in V_0,$ $(\K_X(v))_o\in U_0.$ The image of this mapping, also designated by $B_1,$ is
a well--defined subspace of $N_0$ and thus gives a further orthogonal decomposition of $W$ as $V_0\oplus \mathrm{Im}B_1\oplus \left(V_0\oplus \mathrm{Im}B_1\right)^\perp.$
 Call $N_1=(V_0\oplus \mathrm{Im}B_1)^\perp$ the \emph{first normal subspace}. 
Applying the connection to the second fundamental form at the point $o\in M$ we have $\nabla\K\in S^2T_o^*M\otimes V_0^*\otimes U_0$
(where symmetrisation follows from Gauss--Codazzi equations and flatness of the connection on $\underline W$). If $\pi_1$
denotes the orthogonal projection $\pi_1:W\to N_1$ then $B_2$ is defined as $\pi_1\circ\nabla\K \in S^2T_o^*M\otimes V_0^*\otimes N_1,$
and we have $W=V_0\oplus \mathrm{Im}B_1\oplus \mathrm{Im} B_2\oplus N_2$ where $N_2$ is the \emph{second normal subspace}. 
Recursively, $B_p=\pi_{p-1}\circ\nabla^{p-1}\K\in S^pT^*_oM\otimes V_0^*\otimes N_{p-1}.$
This reiterative process leads to:
\[W=V_0\oplus \mathrm{Im}\,B_1\oplus\mathrm{Im}\,B_2\oplus \dots \oplus \mathrm{Im}\;B_n\oplus N_n\]
If $N_n=0$ this is called the \emph{normal decomposition of $W$ with respect to $V_0.$} \\ \indent
Let us ennunciate without proof two results by the second--named author regarding the normal decomposition which are needed in the sequel
to establish proposition \ref{pre_rigidity}.

\begin{prop}\label{T=Id}
\cite{Na-13} If $W$ is an irreducible $G$--module, then for any $K$--module, $V_0\subset W$ there exists
a positive integer $n$ such that $N_n=0,$ i.e.,
\begin{equation}\label{normal_decomposition} W=V_0\oplus Im\, B_1\oplus\dots\oplus Im\, B_n\end{equation} 
which is a normal decomposition of $(W,V_0).$
\end{prop}
\begin{prop}\label{T=Id2}
\cite{Na-13} Let $W$ be a $G$--module and $V_0\subset W$ a $K$--module. Suppose that $(W,V_0)$ has a normal decomposition.
Assume that each term in the decomposition (\ref{normal_decomposition}) shares no common $K$--irreducible factor with any other
term in the decomposition. Let $T$
be a non--negative Hermitian endomorphism of $W$ which satisfies $(Tgv_1,Tgv_2)=(v_1,v_2)$ for all $g\in G,\;v_1,v_2\in V_0.$ Then, if $T$ is $K$--equivariant, 
$T=Id_{W}.$ 
\end{prop}

A $G$--irreducible representation is a \emph{class--one representation of $(G,K)$}, for $K$ a closed subgroup of $G$ (assumed compact), if it contains non--zero $K$--invariant elements. Then, we can state the following

\begin{prop}\label{pre_rigidity} Let $W=H^0(\C P^1,\mathcal O(k))$ and $V_0$ the $K$--representation regarded 
as the standard fibre for  $\mathcal{O}(k)\to \C P^1.$ Then, 
$\G\mathrm{H}(V_0,V_0)=\mathrm{H}(W).$ 
\end{prop}

\begin{proof}
By Borel--Weil theorem, $W$ is identified with the $\SU(2)$--representation $S^k\C^2$
and, using lemma \ref{subsp}, $V_0$ can be regarded as a subspace of $W.$
The space $W$ decomposes under the $\mathrm{U}(1)$--action as \[W|\mathrm{U}(1)=\C_{-k}\oplus\C_{-k+2}\oplus\dots\oplus\C_k.\] 
Indeed, this is the normal decomposition by proposition \ref{T=Id} where $V_0=\C_{-k}.$\\ \indent
Let $H$ be a class--one submodule of $(G,K)$ in $\mathrm{H}(W)$. 
Suppose that $H\not\subset \mathrm{GH}(V_0,V_0)$. 
Then, by a standard argument, we can assume that $H \bot \mathrm{GH}(V_0,V_0)$. 
Since $H$ is a class--one representation, there exsits a non--zero $C \in H$ such that 
$kCk^{-1}=C$ for all $k \in K$. It follows from the orthogonality assumption that 
\begin{align*}
0= &(C, gH(v_1, v_2))_{\mathrm{H}(W)}=(C,H(gv_1, gv_2))_{\mathrm{H}(W)} \\
= & \frac{1}{2}\left\{(Cgv_1, gv_2)_W+ (Cgv_2, gv_1)_W\right\},
\end{align*}
for arbitrary $g \in G$ and $v_1$, $v_2 \in V_0\subset W$. 
Polarization gives 
$$ 0=(Cgv_1, gv_2), \quad g \in G, \,\, v_1, v_2 \in V_0. $$
If $C$ is sufficiently small, then $Id+C>0$ and so, we can define 
a positive Hermitian operator $T$ satisfying $T^2=Id+C$. 
Then we have 
$$ 
(Tgv_1, Tgv_2)=(v_1,v_2) \quad g \in G, \,\, v_1, v_2 \in V_0.
$$
Since $T$ is also $K$--equivariant, proposition \ref{T=Id2} yields that $T=Id$ and so, $C=0$, 
which is a contradiction.  Hence, every class--one subrepresentation of $(G,K)$ in $\mathrm{H}(W)$
is included in $\mathrm{GH}(V_0,V_0).$ However, it follows from the Clesbsch--Gordan formulae
that $\mathrm{H}(W)$ is composed by class--one representation of $(G,K)$ only, therefore $\G\mathrm{H}(V_0,V_0)=\mathrm{H}(W).$
\end{proof}

\begin{rem}
\emph{A more general version of proposition \ref{pre_rigidity} can be found in \cite{Na-13}, proposition 7.9. Our
proof is essentially the same with the obvious  particularizations.}
\end{rem}

We shall prove  the following interesting result: 
\begin{thm}\label{rigidity}
Let $W=S^{2k}\C^2$ such that $W^\R=S^k_0\R^3\cong \R^{2k+1}.$
If $f:\mathbf CP^1 \to Gr_{2k-1}(\mathbf R^{2k+1})$ is a holomorphic isometric 
embedding of degree $2k,$ then $f$ is the standard map by $W^\R$ up to gauge equivalence.
\end{thm} 

Before proving theorem \ref{rigidity}, let us clarify the construction of the mapping $f:\mathbf CP^1\to Gr_{2k-1}(\mathbf R^{2k+1})$
from the vector bundle viewpoint.

If we regard the complex projective line as the symmetric space $G/K$ where $G=\text{SU}(2)$ and $K=\text{U}(1),$ 
then by Borel--Weil theorem
the space of sections $\Gamma(\mathcal O(2k))$ becomes a $G$--module such that  $W=H^0(\C P^1;\mathcal O(2k))\cong S^{2k}\C^2.$ 
The decomposition of $S^{2k}\C^2$ into irreducible $\text{U}(1)$--modules is as follows: 
\begin{equation}\label{su2_to_u1} S^{2k}\C^2|\text{U(1)} = \bigoplus_{r=0}^{2k} \C_{2k-2r} 
\end{equation}
The typical fibre of $\mathcal O(2k)\to \mathbf CP^1$ is regarded as a subspace $\mathbf C_{-2k}$ in the decomposition by lemma \ref{subsp}. 

Since $W$ has an invariant real structure,  
we have an invariant real subspace denoted by $W^\R=(S^{2k}\C^2)^\R\cong S^{k}_0\R^3$ of real dimension $2k+1.$ 
The real structure descends to the splitting (\ref{su2_to_u1}) but now each irreducible $\text{U}(1)$--module is not invariant under the real structure, 
but
$\sigma(\C_{2k-2r})=\C_{-2k+2r}.$ Therefore for each $r=0,\dots, k$  the space 
$(\C_{2k-2r}\oplus \C_{-2k+2r})$ is stable under the real structure and decomposes in two real isomorphic irreducible $\text{U}(1)$--modules,
denoted by $(\C_{2k-2r}\oplus \C_{-2k+2r})^\R,$ such that (\ref{su2_to_u1}) would be rewritten as
\begin{equation}\label{real}S_0^k\R^3|{\rm U}(1) = \bigoplus_{r=0}^{2k} (\C_{2k-2r}\oplus\C_{-2k+2r})^\R. 
\end{equation}
This implies that $\mathcal O(2k)\to \mathbf CP^1$ 
is globally generated by $W^\R.$ 
Thus, we can define a \emph{real standard map} $f_0:\C P^1 \to Gr_{2k-1}(\R^{2k+1})$ by $W^{\mathbf R}$, 
which turns out to be a holomorphic isometric embedding of degree $2k$ by lemma \ref{stharm}. 
Using the inner product on $W^{\mathbf R}$ and the fibre--metric on $\mathcal{O}(2k)\to \mathbf CP^1,$ 
it is possible to define the
 adjoint of the evaluation which at the identity of $G/K$ determines a mapping $ev^*_{[e]}:\mathcal{O}(2k)\to \underline{W}^\R$ 
whose image is just $(\C_{2k}\oplus\C_{-2k})^\R.$ 

Within this framework we have a real version of proposition \ref{pre_rigidity}, which is the in the core of the proof
of theorem \ref{rigidity}:

\begin{prop} \label{pre_rigidity2} Let $W=H^0(\C P^1,\mathcal O(2k))$ and $V_0$ the $K$--representation regarded 
as the standard fibre for  $\mathcal{O}(2k)\to \C P^1.$ Then, 
$\G\mathrm{S}(V_0,V_0)=\mathrm{S}(W^\R)$. 
\end{prop}

\begin{proof}
Equation (\ref{real}) gives the normal decomposition of $W^\R$ where now $V_0=(\C_{-2k}\oplus\C_{2k})^\R.$ 
The space of symmetric endomporphisms of $W^\R$ can be identified by decomposing first the tensor product
using lemma \ref{CGR}, and identifying the symmetric components
\[\mathrm{S}(W^\R)\subset \mathrm{End}(W^\R)=W^\R\otimes_\R (W^\R)^*\cong \otimes^2 W^\R\]
\[\mathrm{S}(W^\R)=\bigoplus_{r=0}^{k} S^{4k-4r}_0\R^3 \subset \otimes^2 W^\R = \bigoplus_{r=0}^{2k} S^{4k-2r}_0\R^3\]
Notice that all these modules are class--one representations. Then, a similar argument as the one in the proof
of proposition \ref{pre_rigidity} yields the desired result.
\end{proof}

We can now proceed to prove theorem \ref{rigidity}:

\begin{proof}
Consider the real standard map by the holomorphic line bundle $\mathcal{O}(2k)\to \C P^1$
and $W^{\mathbf{R}}$ as depicted above. 
Therefore by proposition \ref{pre_rigidity2}, $\mathrm{S}(W^\R) = \G\mathrm{S}(V_0,V_0).$
and replacing $\R^{n+2}$ by $W^\R$ in theorem \ref{GenDW} the real standard map admits no deformations. 
\end{proof}

\section{Moduli space by gauge equivalence}

We undertake now the task of giving an accurate  description of the 
moduli space of holomorphic isometric embeddings $\C P^1\to Gr_p(W)$ up to gauge equivalence. 
Our strategy will be to capitalise on the representation--theoretic formulae of Sections 4 and 5
to explicitly determine the subspaces of linear operators in $\mathrm{S}(W)$ which specify
the moduli. Such subspaces are sharply characterised by condition (III) in theorem \ref{GenDW}. 
This is achieved after a sequence of step--stone results culminating in lemma \ref{intersection} and its Corollary, which
allows to compute the moduli dimension.\\ \indent
As indicated by condition (IV) in theorem \ref{GenDW}, the gauge equivalence relation is to be taken into account to 
obtain
the moduli space and to give a geometric meaning to its compactification in the natural $L^2$--topology.
A qualitative description of these spaces is given in theorem \ref{gmod1}.\\

Let $W$ be the space of holomorphic sections of $\mathcal O(k) \to \mathbf CP^1$ which, 
by Borel--Weil theorem, is identifid with the $\text{SU}(2)$--representation $S^k\mathbf C^2.$ Equation
(\ref{su2_to_u1}) gives a a weight decomposition of $W$ with respect to ${\rm U}(1).$ 
When $\mathcal O(k) \to \mathbf CP^1$ is regarded as the homogeneous line bundle 
$\text{SU}(2)\times_{\text{U}(1)} V_0 \to \mathbf CP^1$, then $V_0$ is identified with the ${\rm U}(1)$--irreducible
subspace  $\mathbf C_{-k}$ of $W$ by lemma \ref{subsp}. \\ \indent
In order to apply theorem \ref{GenDW} we shall regard the quotient bundle as a real vector bundle of rank $2.$
Following the generalizaion of do Carmo--Wallach theory, we must determine the subspaces 
$\G\mathrm{S}(V_0, V_0)$ and $\G\mathrm{S}(\mathfrak{m}V_0,V_0)$ of $\mathrm{S}(W).$ \\ \indent
From now on $V_0$ and $W$ 
shall stand either for the complex modules or for their underlying $\R$--vector spaces whenever
the meaning is clear, avoiding the heavier notation $(V_0)_\R$ or $W_\R;$  In the remaining sections we will
adopt this convention.\\

Since $\G\mathrm{H}(V_0,V_0)$ is a proper subspace of $\G\mathrm{S}(V_0,V_0)$, we have that 
$\mathrm{H}(W) \subset \G\mathrm{S}(V_0,V_0).$ 
We must determine the intersection between $\G\mathrm{S}(V_0,V_0)$ and subspaces $\sigma\mathrm{H}_+(W)\oplus J\sigma\mathrm{H}_+(W)$
appearing in Corollary \ref{DecSYM}.
The same is true for  the intersection $\G\mathrm{S}(\mathfrak{m}V_0,V_0)$ with $\sigma\mathrm{H}_+(W)\oplus J\sigma\mathrm{H}_+(W)$ 
as we shall consider immediately. 

\begin{lemma}
$\mathfrak{m}V_0=\mathbf C_{-k-2}$. 
\end{lemma}

\proof
By the decomposition of $S^2\C^2$ into ${\rm U}(1)$--irreducible modules $S^2\C^2|{\rm U}(1)=\C_{2}\oplus\C_0\oplus \C_{-2}$
and using the real structure we have 
$(S^2\C^2)^\R\cong \mathfrak{su}(2),$ $(\C_0)^\R\cong \mathfrak{u}(1)$ therefore
$( \C_{2}\oplus\C_{-2})^\R\cong \mathfrak{m}.$ Then,
\[ \mathfrak m \otimes  V_0 = (\C_{2}\oplus \C_{-2})\otimes \C_{-k} = \C_{-k+2}\oplus\C_{-k-2}\]
The action of $\mathfrak m$ on $V_0$ is then obtained by projecting $\mathfrak m \otimes V_0$ back to $S^k\C^2;$
therefore
\[\mathfrak{m}V_0= \left(\mathfrak{m}\otimes V_0 \right) \cap S^k\C^2|{\rm U}(1) = \C_{-k+2}\]
 \qed

\begin{lemma}\label{intersection}
$\G\mathrm{S}(\mathfrak{m}V_0,V_0)\cap \sigma \mathrm{H}_+(W) \oplus J\sigma \mathrm{H}_+(W)$ is 
the highest weight representations of $\text{SU}(2)$ appeared in proposition \ref{DecH}. 
\end{lemma}

\proof
Let $u_{-k}$ and $u_{-k+2}$ be respectively unitary bases for the complex one--dimensional ${\rm U}(1)$--modules
$V_0=\C_{-k}$ and $\mathfrak{m}V_0=\C_{-k+2}.$ Then, The space $\mathrm{H}(\mathfrak{m}V_0,V_0)\equiv \mathrm{H}(\C_{-k+2},\C_{-k})$
is the complex span of \[2 \mathrm{H}(u_{-k+2},u_{-k})= u_{-k+2}\otimes(\cdot,u_{-k})_{_W} + u_{-k}\otimes (\cdot,u_{-k+2})_{_W}\] 
where $(,)_{_W}$ denotes the Hermitian inner product on $S^k\C^2.$
When $\C_{-k}$ and $\C_{-k+2}$ are regarded as their underlying two--dimensional $\R$--vector spaces $\R^2_{k}$ and $\R^2_{k-2},$ 
real bases are given respectively by $\{u_{-k},Ju_{-k}\}$ and $\{u_{-k+2},Ju_{-k+2}\}$ 
where $J$ is the almost complex structure induced by the multiplication by the imaginary unit. Using these real bases the complex form
 $2 \mathrm{H}(u_{-k+2},u_{-k})$ can be rewriten as a real operator
\begin{eqnarray*}
2 \mathrm{H}(u_{-k+2},u_{-k})|\R & = & u_{-k+2}\otimes\langle \cdot,u_{-k}\rangle_{_W} + Ju_{-k+2}\otimes\langle \cdot,Ju_{-k}\rangle_{_W} \\
                              &  &  + u_{-k}\otimes\langle \cdot,u_{-k+2}\rangle _{_W} + Ju_{-k}\otimes\langle \cdot,Ju_{-k+2}\rangle_{_W}
\end{eqnarray*} where $\langle,\rangle_{_W}$ is the inner  product on $W_\R$ induced from the Hermitian inner product on $W.$
Write the basis for $\mathrm{S}(\mathfrak{m}V_0,V_0)\equiv \mathrm{S}(\R^2_{-k+2},\R^2_{-k})$ as
$\{ \mathrm{S}(u_{-k+2},u_{-k}),$ $\mathrm{S}(Ju_{-k+2},u_{-k}),$ $\mathrm{S}(u_{-k+2},Ju_{-k}),$ $\mathrm{S}(Ju_{-k+2},Ju_{-k})   \},$ eg.,
\[2\mathrm{S}(u_{-k+2},u_{-k})= u_{-k+2}\otimes \langle \cdot, u_{-k}\rangle_{_W} + u_{-k}\otimes\langle\cdot,u_{-k+2}\rangle_{_W},\quad etc.\] Comparing both equations we have:
\[\mathrm{H}(u_{-k+2},u_{-k})|\R = \mathrm{S}(u_{-k+2},u_{-k})+\mathrm{S}(Ju_{-k+2},Ju_{-k}).\] Analogously,
\[\mathrm{H}(u_{-k+2},i u_{-k})|\R = \mathrm{S}(u_{-k+2},Ju_{-k})-\mathrm{S}(Ju_{-k+2},u_{-k}).\]
 Let us define a new elements $\{X,Y\}$
\[ \begin{array}{ccl}
         X & = &  \mathrm{S}(u_{-k+2},u_{-k})-\mathrm{S}(Ju_{-k+2},Ju_{-k})\\
         Y & = &  \mathrm{S}(u_{-k+2},Ju_{-k})+\mathrm{S}(Ju_{-k+2},u_{-k})
         \end{array} \] 
 $X,Y\in\mathrm{S}(W_\R)$ are orthogonal to the subspace of Hermitian matrices $\mathrm{H}(W)\subset \mathrm{S}(W_\R),$ therefore
they belong to $\sigma H_+(W)\oplus J\sigma H_+(W)$ according to Corollary \ref{DecSYM}.\\ \indent
Let us consider the contragredient action of the structure map $\sigma$ on $X,$ the case of $Y$ being analogous.  Firstly, 
\[
\sigma \left( u \otimes \langle\cdot,v\rangle_{_W}\right)\sigma = \sigma u \otimes \langle \sigma \cdot,v\rangle_{_W} = \sigma u\otimes \langle\cdot,\sigma v\rangle_{_W}
\]
and as such $\sigma \mathrm{S}(u,v)\sigma=\mathrm{S}(\sigma u, \sigma v).$\\ \indent
Secondly, the ${\rm U}(1)$--modules $\C_{i}$ are not $\sigma$--invariant but $\sigma (\C_{\pm i})=\C_{\mp i},$ for all $i$ that is, $\sigma u_{\pm i} =u_{\mp i}.$ which, together with conjugate--linearity of the
structure map yields $\sigma(\R^2_{\pm i})=\R^2_{\mp i}:\{u_{\pm i},Ju_{\pm i}\}\mapsto \{u_{\mp i},-Ju_{\mp i}\}.$ Hence we have:
\begin{eqnarray*}
X^\sigma=\sigma X \sigma & = & \mathrm{S}(\sigma u_{-k+2},\sigma u_{-k})- \mathrm{S}(\sigma Ju_{-k+2},\sigma J u_{-k})\\
                & = & \mathrm{S}(u_{k-2},u_k)-\mathrm{S}(Ju_{k-2},Ju_{k})
\end{eqnarray*} This is not an element of $\mathrm{S}(\mathfrak m V_0,V_0)\equiv \mathrm{S}(\R^2_{-k+2},\R^2_{-k})$ but $X^\sigma\in \mathrm{S}(\R^2_{k-2},\R^2_k).$ 
Note that we can find $g\in \SU(2)$ such that  $\mathrm{S}(u_{k-2},u_{k})=
\mathrm{S}(gu_{-k+2},gu_{-k})=g\cdot \mathrm{S}(u_{-k+2},u_{-k})\in \G\mathrm{S}(\mathfrak m V_0,V_0)$ 
up to a sign. 
\\ \indent Let us add
$  Y^\sigma =   \mathrm{S}(u_{k-2},Ju_k) + \mathrm{S}(Ju_{k-2},u_k)$ for the sake of completeness.\\
The preceding argument also shows that $\{$ $S(u_{k-2},u_{k}),$ $S(u_{k-2},Ju_k),$ $S(Ju_{k-2},u_k),$ $S(Ju_{k-2},Ju_k)$ $\}$ 
spans a subspace of $\G\mathrm{S}(\mathfrak{m}V_0,V_0).$\\

Morover, using the characterisation given in Corollary \ref{DecSYM} we have:
\[
X+X^\sigma \in \sigma \mathrm H_+(W), \qquad  X-X^\sigma \in J\sigma\mathrm H_+(W)
\] 
The same inclusions are also true for $Y\pm Y^\sigma.$\\ \indent
From the expression of the 
action of $\sigma$ on $\mathrm{H}(u,v)$
\begin{eqnarray*}
\sigma\cdot \mathrm{H}(u,v) & = &  \sigma\left(u\otimes (\cdot,v)_{_W}+v\otimes(\cdot,u)_{_W}\right)\sigma =  \sigma u \otimes (\sigma\cdot,v)_{_W}+\sigma v \otimes (\sigma\cdot,u)_{_W}\\
                              & = & \sigma u \otimes\overline{(\cdot,\sigma v)_{_W}} + \sigma v \otimes \overline{(\cdot,\sigma u)_{_W}}
\end{eqnarray*}
it is easy to write $X\pm X^\sigma$ back in terms of Hermitian operators as
\[X\pm X^\sigma = \sigma \cdot \left( \mathrm{H}(u_{k-2},u_k)\pm \mathrm{H}(u_{-k+2},u_{-k})\right) | \R\] 
The toral action of a ${\rm U}(1)$--element of $\SU(2)$ on $u_{\pm k},\;u_{\pm (k- 2)}$ yields
\[ \exp(i\theta)u_{\pm k}= \exp(\pm i k \theta) u_{\pm k},\qquad   \exp(i\theta)u_{\pm (k- 2)}= \exp(\pm i (k-2) \theta) u_{\pm (k- 2)}\]
and as such, $X\pm X^\sigma$ (considered as the Hermitian operator above) contains terms of weight $\pm(2k-2).$ However, 
from Corollary \ref{DecH} we know that the only component in the real decomposition of $\sigma\mathrm{H}_+(W)$ and $J\sigma\mathrm{H}_+(W)$
(both isomorphic to $\mathrm{H}_+(W)$) which can host  such a vector is the top term $S^k_0\R^3$ on each space. Therefore
\[ \G\mathrm{S}(\mathfrak{m}V_0,V_0) \cap \sigma \mathrm{H}_+(W) = S^k_0\R^3\qquad (\;resp.\;for\;J\sigma\mathrm{H}_+(W)\;)  \]
And as a result
\[ \G\mathrm{S}(\mathfrak{m}V_0,V_0) = \mathrm{H}(W) \oplus S^k_0\R^3 \oplus S^k_0\R^3\] \qed

In other words:

\begin{cor}\label{intersection 2}
The orthogonal complement to $\G\mathrm{S}(\mathfrak{m}V_0,V_0) \oplus \mathbf{R}\,Id$
in $\mathrm{S}(W)$ is 
\[ 2\bigoplus_{r=1}^{k\geq 2r} S^{k-2r}_0\mathbf R^3\]
\end{cor}

This follows from applying the previous lemma to the explicit expressions for the components of
$\mathrm{S}(W)$ as described in proposition \ref{DecH}, and accounts for the space of symmetric 
operators $T$ described by the second relation in (\ref{HDW 2}), i.e.,  condition $\mathrm{III}$ in theorem \ref{GenDW}.

\begin{rem}
\emph{The first condition in (\ref{HDW 2}) is for all our pourposes inessential: Let} $\G\mathrm{S}_0(V_0,V_0)$ \emph{be the orthogonal
 complement of the $G$--invariant, irreducible 
submodule generated by the identity in} $\G\mathrm{S}(V_0,V_0)$. 
\emph{We denote by} $\text{S}_0(W)$ \emph{the set of tracefree symmetric operators on} $W$ 
\emph{with the induced inner product from} $\text{S}(W)$. 
\emph{Then,}
\[ \G\mathrm{S}_0(V_0,V_0) \subset \G\mathrm{S}(\mathfrak{m}V_0,V_0),\]
\emph{which stems from an analogous result to lemma \ref{intersection} applied to} $\G\mathrm{S}_0(V_0,V_0).$ \emph{The
proof is equivalent, changing the wheight} $\pm(2k-2)$ \emph{by} $\pm 2k$ \emph{in the curcial final step.}
\end{rem}

Condition $\mathrm{III}$ in theorem \ref{GenDW} is fulfilled by the family of operators in Corollary \ref{intersection 2} (see remark above)
thus accounting for all holomorhic embeddings $f:\C\mathbf P^1\to Gr_p(\mathbf{R}^m)$ up to possible degeneracies. Quantitative information
about the moduli (i.e., its dimension) can therefore be derived from the Corollary:
\begin{equation}
{\rm dim}_\R\,\mathcal M_k  = k(k-1).
\end{equation}
The following theorem summarises the qualitative information about
the moduli space and gives a neat geometric interpretation to its compactification.

\begin{thm}\label{gmod1}  
If $f:\C P^1\to Gr_n(\R^{n+2})$ is a full holomorphic isometric embedding of degree $k,$ then $n\leq 2k.$

Let $\mathcal M_k$ be the moduli space of full holomorphic isometric embeddings of degree $k$
  of the complex projective line into $Gr_{2k}(\mathbf R^{2k+2})$
 by the gauge equivalence of maps. 
Then, $\mathcal M_k$ can be regarded as an open bounded convex body 
in $2\bigoplus_{r=1}^{k\geqq 2r}S^{k-2r}_0\mathbf R^3$.

Let $\overline{\mathcal M_k}$ be the closure of the moduli $\mathcal M_k$ by the inner product. 
boundary points of $\overline{\mathcal M_k}$ describe those maps whose images 
are included in some totally geodesic submanifold $Gr_p(\mathbf R^{p+2})$ of 
$Gr_{2k}(\mathbf R^{2k+2}),$ where $p<2k.$\\ \indent 
The totally geodesic submanifold $Gr_p(\mathbf R^{p+2})$ 
can be regarded as the common zero set of some sections of $Q\to Gr_{2k}(\mathbf R^{2k+2})$,
which belongs to $\mathbf R^{2k+2}$.
\end{thm} 

\begin{proof}
The restriction $n\leq 2k$ follows from I in theorem \ref{GenDW} and Borel--Weil theorem.

It is evident from III 
in theorem \ref{GenDW}  that $\G\mathrm{S}(\mathfrak{m}V_0,V_0)^\perp$ is a parametrisation
of the space of full holomorphic isometric embeddings $f: \C\mathbf{P}^1\to Gr_{2k}(\mathbf R^{2k+2})$ of degree $k.$ 
Positivity of $T$ being garanteed by fullness, we can apply the original do Carmo--Wallach argument
 \cite{DoC-Wal}, \S 5.1, to conclude that $\mathcal M_k$ is a bounded connected \emph{open} convex body in  $\mathrm{H}(W)$ with the
topology induced by the $L^2$ scalar product.

Under the natural compactification in the $L^2$--topology, the boundary points correspond to operators $T$ which are not positive,
but semipositive. It follows from  IV in theorem \ref{GenDW} that each of these operators defines in turn
a full holomorphic isometric embedding $\C\mathbf{P}^1\to Gr_{p}(\mathbf R^{p+2}),$ of degree $k$ with $p=2k-{\rm dim}\;\mathrm{Ker}\;T,$ whose target embeds in $Gr_{2k}(\R^{2k+2})$
as a totally geodesic submanifold. The image $Z$ of the embedding $Gr_p(\R^{p+2})\to Gr_{2k}(\R^{2k+2})$ is determined by the common zero--set of
sections in  $\mathrm{Ker}\;T.$
\end{proof}

\section{Moduli space by image equivalence}

Remember that $S^{k}_0\mathbf R^3,$ defined in lemma \ref{rltnshp},  is the real invariant part of 
$S^{2k}\mathbf C^2$ which is the symmetric power of the standard representation 
$\mathbf C^2$ of $\text{SU}(2).$ The moduli space 
$\mathcal M_k$   
has a natural complex structure induced by that on $Q\to Gr_{2k}(\R^{2k+2})$
which coincides with the one introduced in Remark \ref{complex}, $\S 5.$  Hence, $\mathcal M_k$ can
be regarded as 
holomorphically included in the $\C$--vector space
$\oplus _{r=1}^{k\geqq 2r}S^{2k-4r}\mathbf C^2$.  
We can show that the centralizer of the holonomy group acts on $\mathcal M_k$ 
with weight $-k$. 
Hence we have 

\begin{thm}\label{imod}
Let $\mathbf M_k$ be the moduli space of holomorphic isometric embeddings 
of the complex projective line into $Gr_{2k}(\mathbf R^{2k+2})$ 
of degree $k$ 
by the image equivalence of maps. 
Then we have $\mathbf M_k=\mathcal M_k\slash S^1$. 
\end{thm}

\begin{proof}
Assume two full holomorphic isometric embeddings $\C P^1\to Gr_{2k}(\R^{2k+2})$ of degree $k$ to be image equivalent. They may
represent distinct points in $\mathcal M_k.$ 
By definition of image equivalnce, there is an isometry $\psi$ of $Gr_{2k}(\R^{2k+2})$ such that $f_2=\psi\circ f_1,$ then $f_2^*Q=f_1^*\tilde\psi Q$ as sets.
Using the  natural identifications $\phi_1,\phi_2$ of $\mathrm{IV}$ in theorem \ref{GenDW} we introduce new bundle isomorphisms $\mathcal{O}(k)\to f_2^*Q$ defined by $\tilde \psi\circ\phi_1$
and $\phi_2.$ Hence,  we have a gauge transformation $\phi_2^{-1}\tilde\psi\phi_1$ on the line bundle $\mathcal{O}(k)\to M$ 
preserving the metric and the connection.
By connectedness of $\C P^1$ such a gauge transformation is regarded as an element of the centralizer of the holonomy group of the connection in the sructure group of $V$ i.e., 
${\rm U}(1)\equiv S^1$ acting with weight $-k$ on the standard fibre  $V_0\cong\C_{-k}.$ 
Modding out the $S^1$--action yields the  true moduli space by image equivalence $\mathbf M_k.$
\end{proof}

\begin{rem}
$\mathcal M_k$ \emph{has a complex structure (see remark in \S 5) and a metric induced by the inner product both preserved by the $S^1$--action. 
Hence, it is a K\"ahler manifold together with a $S^1$--action preserving the K\"ahler structure. Therefore,}  $\mathcal M_k$ 
\emph{ is
naturally equipped with a moment map} $\mu:\mathcal M_k\to \R$ \emph{expressed as} $\mu=|T|^2.$
\end{rem}

\begin{cor}
There exists a one--parameter family $\{f_t\},\;t\in[0,1],$ of $\SU(2)$--equivariant image--inequivalent holomorphic isometric embeddingds of even degree  of 
$\C P^1$ into complex quadrics where $f_0$ corresponds to the standard map and $f_1$ is the real standard map.
\end{cor}

\begin{proof}
The moduli space by gauge--equivalence  $\mathcal M_k$ is contained in $\oplus_{r=1}^{k\geq 2r}S^{2k-4r}\C^2.$ For even $k$ this last expression
includes the trivial representation $\C,$ which using the real structure can be described as $\C=\R\sigma \oplus \R J\sigma$. 
Let $C\in \C\subset \oplus_{r=1}^{k\geq 2r}S^{2k-4r}\C^2.$ If it is small enough, then by theorem 2.4, $Id+C$ determines a
holomorphic isometric embedding into $Gr_{2k}(\R^{2k+2}).$ The group $\SU(2)$ acts on each component of $Id+C$ trivially, so the associated 
holomorphic isometric embedding is $\SU(2)$--equivariant. The $S^1$ action of the centralizer of the holonomy group  
acts on $\C$  with weight $-k$ (see proof of theorem \ref{imod}) therefore, taking quotient by the $S^1$--action,
we obtain a half--open segment parametrising the described  maps, which becomes a closed segment under the natural 
compactification in the $L^2$--topology. 
Let $C=t\sigma + sJ\sigma$. 
Then we can show that $Id+C$ is positive if and only if $t^2+s^2<1$. 
Suppose that $t^2+s^2=1$. 
Then $(t+sJ)\sigma$ is also an invariant real structure on $S^{2k}\mathbf C^2$. 
Hence we may consider only the case that $t=1$ and $s=0$. 
Since the kernel of $Id+\sigma$ is $JW^{\mathbf R}$, 
theorem \ref{GenDW} implies that $Id+\sigma$ determines a totally geodesic submanifold 
$Gr_{2k-1}(\mathbf R^{2k+1})$ of $Gr_{4k}(\mathbf R^{4k+2})$ and a holomorphic isometric embedding into 
the submanifold $Gr_{2k-1}(\mathbf R^{2k+1})$ represented by $2\,Id_{W^{\mathbf R}}$. 
This map is nothing but the real standard map by $W^{\mathbf R}$, 
because constant multiple of the identity gives the same subspace of $W^{\mathbf R}$. 
\end{proof}

\end{document}